\documentclass{article}
\usepackage[utf8]{inputenc}
\usepackage{amsfonts}
\usepackage{amssymb}
\usepackage{tikz-cd}
\usepackage{mathtools}
\usepackage{amsmath}
\usepackage{graphicx}
\usepackage{commath}
\graphicspath{images/}
\usepackage[english]{babel}
\usepackage{dirtytalk}
\usepackage{amsthm}
\usepackage{cancel}
\usepackage{authblk}
\usepackage{hyperref}
\usepackage{verbatim}

\usepackage[maxbibnames=99]{biblatex}
\newtheorem{Theorem}{Theorem}[section]

\newtheorem{lemma}[Theorem]{Lemma}
\newtheorem{proposition}[Theorem]{Proposition}
\newtheorem{corollary}[Theorem]{Corollary}

\newtheorem{remark}[Theorem]{Remark}

\makeatletter
\newcommand{\tpitchfork}{%
  \vbox{
    \baselineskip\z@skip
    \lineskip-.52ex
    \lineskiplimit\maxdimen
    \m@th
    \ialign{##\crcr\hidewidth\smash{$-$}\hidewidth\crcr$\pitchfork$\crcr}
  }%
}
\makeatother

\title{Uniqueness of Lagrangians in $T^*\mathbb{R}P^2$}
\author{Nikolaos Adaloglou\footnote{ This scientific paper was supported by the Onassis Foundation - Scholarship ID: F ZP 001-1/2019-2020}  \\ \href{mailto:n.nadaloglou@math.leidenuniv.nl}{ n.adaloglou@math.leidenuniv.nl}}

\begin{document}
\maketitle

\begin{abstract}
    We present a new and simpler proof of the fact that any Lagrangian $\mathbb{R}P^2$ in $T^*\mathbb{R}P^2$ is Hamiltonian isotopic to the zero section. Our proof mirrors the one given by Li and Wu for the Hamiltonian uniqueness of Lagrangians in $T^*S^2$, using surgery to turn Lagrangian spheres into symplectic ones. The main novel contribution is a detailed proof of the folklore fact that the complement of a symplectic quadric in $\mathbb{C}P^2$ can be identified with the unit cotangent disc bundle of $\mathbb{R}P^2$.
\medskip


\end{abstract}


\section*{Introduction}
The Nearby Lagrangian Conjecture, due to Arnol'd, states that in a cotangent bundle, any exact Lagrangian is Hamiltonian Isotopic to the zero section. Even for surfaces, this conjecture is generally open, with results obtained only for $T^*S^2$ by Hind in \cite{H}, in $T^*\mathbb{R}P^2$ by Hind-Pinsonnault-Wu in \cite{HPW} and by Dimitroglou Rizell-Goodman-Irvii in $T^*T^2$ in \cite{DGI}.\par

In \cite{LW}, Li-Wu give an alternative proof of Hamiltonian uniqueness in $T^*S^2$. They use only compact $J$-holomorphic curves and avoid the more advanced techniques of Symplectic Field Theory, which were implemented originally by Hind. The goal of this paper is to use the ideas of Li-Wu to give an elementary proof for Hamiltonian uniqueness in $T^*\mathbb{R}P^2$. The proof we present here is completely analogous to the proof of Theorem 6.6 in \cite{LW}. The novel contribution of this paper is a detailed proof of the folklore fact that the (unit) cotangent bundle of $\mathbb{R}P^2$ compactifies to $\mathbb{C}P^2$. \par

The paper is organized as follows: First, we recall the symplectic blow-up and down procedure and how it is used to obtain the main hard results we use, namely Lemma \ref{sphereuniq} and Theorem \ref{laguniqcp2}. They concern the Hamiltonian uniqueness of certain submanifolds in $\mathbb{C}P^2$. In the second section, we explain the compactification which, roughly speaking, allows us to identify the complement of a quadric $Q$ in $\mathbb{C}P^2$ with $T^*\mathbb{R}P^2$ while also identifying the standard Lagrangian $\mathbb{R}P^2$ in $\mathbb{C}P^2$ with the zero-section of $T^*\mathbb{R}P^2$. In the final section, we recall a result of Evans \cite{Emcg} about the symplectic mapping class group of $T^*\mathbb{R}P^2$ and put everything together to finally conclude the proof of Arnol'ds conjecture for $T^*\mathbb{R}P^2$.\par

Let us remark that since the appearance of this paper on the Arxiv, the presented results have been greatly streamlined. Indeed, in her preprint \cite{johbim}, Bimmermann constructs a similar but more natural compactification, where all the symplectic forms are the desired ones and there is no need to appeal to hard pseudoholomorphic techniques. In another direction, Hauber pointed out to us that similar results were independently obtained in \cite{sackel2022certain}.

\subsection*{Acknowledgement}
This paper summarizes the results of the author's Master Thesis which was carried out under the supervision of the late Will Merry. Will's choice of topic and overall guidance deeply shaped my mathematical understanding and are still visible years after my time at ETH. This paper is dedicated to him.

\section{Recollection on the Symplectic Rational Blow-up and down}
The rational blow-up and blow-down are surgeries performed in $4$-manifolds, originally introduced by Fintushel-Stern in \cite{FS}. Roughly speaking, the idea is to look at certain configurations of spheres $C=\cup S_i$ and to consider different fillings of the 3-manifold $Y$ that appears as the boundary of a tubular neighborhood of $C$. For instance, different fillings may come from viewing $Y$ as the link of a smoothable surface singularity. Considering $C$ as the exceptional divisor of the resolution, another filling is obtained by the Milnor fiber (or smoothing) of the singularity. In fact, the usual Blow-up can be seen as a trivial instance of the above idea, where the isolated singularity is already smooth.\par

Rationally blowing-up/down can also be done in a symplectic manner under suitable conditions, as shown first by Symington \cite{Sym}. The reader is also referred to \cite{evans_2023} for a more modern exposition. For us, the relevant cases are when $C$ consists of a single $(+2)$ or $(+4)$ symplectic sphere, and blowing down $C$ will replace it with a Lagrangian $S^2$ or $\mathbb{R}P^2$ respectively. Conversely, blowing up such a Lagrangian $L$ will replace it with the corresponding symplectic sphere $C$. The power of this technique comes from the following simple observation with many surprising applications.

\begin{proposition}\emph{(\cite{LW},\cite{BLW})}\label{rblowup}
Let $L$ be either a Lagrangian $S^2$ or $\mathbb{R}P^2$ in a symplectic $4$-manifold $(M,\omega)$. Rationally blowing up $L$ produces a new symplectic manifold $(\widetilde{M},\widetilde{\omega})$ where $L$ has been replaced by $S$, a symplectic $(-2)$ or $(-4)$-sphere respectively. In addition, if $M$ is rational, then so is $\widetilde{M}$.
\end{proposition}

\begin{remark}
Since blowing-up along $L$ alters $M$ only in a small neighborhood of $L$, it is clear that $(M-\nu L,\omega)$ is symplectomorphic to $(\widetilde{M}-\nu S,\widetilde{\omega})$, where $\nu L$ and $\nu S$ are small enough standard Weinstein neighborhoods of the corresponding submanifolds.
\end{remark}

One instance where the above remark becomes useful is that symplectic behavior \textit{relative} to $L$ is the same as relative to $S$. For example, in section 6.4.1 in \cite{LW}, the following is shown:

\begin{lemma}\label{sphereuniq}
Any two symplectic $(+4)$-spheres in $(\mathbb{C}P^2,\omega_{FS})$ disjoint from the standard $\mathbb{R}P^2$ are Hamiltonian isotopic via an isotopy that fixes the Lagrangian $\mathbb{R}P^2$.
\end{lemma}

Another and arguably more important, application of Proposition \ref{rblowup} is the proof of Hamiltonian uniqueness of spheres or projective spaces in certain symplectic rational manifolds.

\begin{Theorem}\emph{(\cite{H}, Section 6 in \cite{LW}, Theorem 1.3 in \cite{BLW})}\label{laguniqcp2}
Any two Lagrangian $\mathbb{R}P^2$s in $(\mathbb{C}P^2,\omega_{FS})$ are Hamiltonian isotopic.
\end{Theorem}

\section{Compactification of the cotangent disc bundle}\label{sectioncompactification}

Now we move on with describing the compactification, following ideas outlined in \cite{AuLag}. First, we will prove a general statement about the compactification of the cotangent disc bundle of $\mathbb{R}P^n$ to $\mathbb{C}P^n$. However, this compactification endows $\mathbb{C}P^n$ with a symplectic form different from the Fubini-Study one. We will then continue to show that for $n=2$, the aforementioned form on $\mathbb{C}P^2$ is indeed symplectomorphic to the standard Fubini-Study form, in a controlled way.

Let us first establish some notation and conventions. We will identify the Cotangent bundle of the Sphere $(T^*S^n,\omega_{std})$ with  the subset of $(\mathbb{C}^{n+1},\omega_{std})$ defined as

\[\{(p,q)\big||p|=1,\langle p,q\rangle=0\}\subset \mathbb{R}^{n+1}\oplus\mathbb{R}^{n+1}\approx \mathbb{C}^{n+1}\]
and we will denote with $U^*_rM$ (resp. $S^*_rM)$ the radius $r$ cotangent disc (resp. sphere) bundle of $M=S^2$ or $M=\mathbb{R}P^2$. In particular, $U^*_1S^2$ can be seen as the tuples of $(p,q)\in T^*S^n$ where $|q|=1$. The boundary of $U_1^*S^1$ will be denoted by $S_1^*S^n$.\par
When we refer to the quadric $Q^n$ in $\mathbb{C}P^{n+1}$ we will always mean the standard Fermat quadric
\[Q^n=[z_0^2+z_1^2+\cdots +z_{n+1}^2=0].\]

Finally, $H_i\subset \mathbb{C}P^n$ will denote the hyperplane with vanishing $i$-th coordinate. By slightly abusing notation we will also denote with $Q^{n-1}$ the intersection $Q^{n}\cap H_{n+1}$, whenever no confusion can arise.

The idea for the general compactification is simple; first, recall that the ball of $\mathbb{C}^{n}$ may be symplectically identified with the complement of a hyperplane $H_{n}$ in $\mathbb{C}P^n$. The (symplectic) inclusion $T^*S^n\hookrightarrow \mathbb{C}^{n+1}$ places $U_1^*S^n$ inside a ball of $\mathbb{C}^{n+1}$ such that the boundary of $U_1^*S^n$ is a subset of the boundary of the ball. Symplectically embedding the ball inside $\mathbb{C}P^n$ will identify $U_1^*S^n$ with the quadric $Q^n$ and the boundary of $U_1^*S^n$ will be mapped to $Q^n\cap H_{n+1}$, realizing the aforementioned identification in a symplectic manner.\par

Let us provide the details: As a warm-up, we first explain how to symplectically embed a ball in $\mathbb{C}P^{n+1}$. The idea is to view the ball in question as the $2(n+1)$-cell that is attached to $\mathbb{C}P^n$ to create $\mathbb{C}P^{n+1}$.
\begin{lemma}\label{projemb}
The open ball of radius $r$, $B^{2n+2}(r)\subset (\mathbb{C}^{n+1},\omega_{std})$, is symplectomorphic to $\mathbb{C}P^{n+1}-H_{n+1}$ with the form $r^2\omega_{FS}$.
\end{lemma}

\begin{proof}
Consider the map $\phi:B^{2n+2}(r)\rightarrow \mathbb{C}P^{n+1}-H_{n+1}$ given by

\[\phi(z_0,\dots z_n)=[z_0:\dots:z_n:i\sqrt{r^2-\sum |z_i|^2}].\]
The map $\phi$ can be factorized through maps $\widetilde\phi:B^{2n+2}(r)\hookrightarrow S^{2n+3}(r)\subset \mathbb{C}^{n+2}$
and $\pi_r:S^{2n+3}(r)\rightarrow \mathbb{C}P^{n+1}$, where $\widetilde\phi$ is just the map
 
\[\widetilde\phi(z_0,\dots z_{n})=(z_0,\dots,z_{n},i\sqrt{r^2-\sum |z_i|^2})\]
and $\pi_r$ the standard quotient by the scalar action. We also have the natural inclusion, $i_r:S^{2n+3}(r)\hookrightarrow \mathbb{C}^{n+2}$, which satisfies 

\[r^2\pi_r^*\omega_{FS}=i^*_r\omega_{std}.\]
 Here we abuse notation slightly, to denote with $\omega_{std}$ the standard symplectic forms both on $\mathbb{C}^{n+1}$ and $\mathbb{C}^{n+2}$ but it should be clear which one we mean. With that in mind, we compute that

\[\phi^*\omega_{FS}=\widetilde{\phi}^*\pi_r^*\omega_{FS}=\frac{1}{r^2}\widetilde{\phi}^*i_r^*\omega_{std}=\frac{1}{r^2}\omega_{std}.\]
The last equality holds since the coordinate $z_{n+1}$ in the image of $\widetilde{\phi}$ is purely imaginary.
\end{proof}

Now, notice that $U^*_1S^n$ is a subset of $B^{2n+2}(\sqrt{2})$. As we showed that $B^{2n+2}(\sqrt{2})$ symplectically compactifies to $\mathbb{C}P^{n+1}$, one can hope that $U^*_1S^n$  as a subset of the ball will also be symplectically compactified. Indeed, that is the case:

\begin{lemma}\label{sphereembedding}
$(U^*_1S^n,\omega_{std})$ is symplectomorphic to $Q^n-Q^{n-1}$ with the restriction of the form $2\omega_{FS}$.
\end{lemma}

\begin{proof}
We will show that the map $\phi$ defined in the previous lemma provides the claimed symplectomorphism.
For $[z_0:\dots:z_{n}:z_{n+1}]\in \mathbb{C}P^{n+1}-H_{n+1}$ we have
\begin{align*}
   z_0^2+\dots+z_n^2+z_{n+1}^2&=z_0^2+\dots+z_n^2-2+\sum_0^n|z_k|^2\\
   &=\sum_0^n (x_k^2-y_k^2)+2i\sum_0^n  x_ky_k-2+\sum_0^n(x_k^2+y_k^2)\\
                                   &=2(\sum_0^nx_k^2-1)+2i\sum_0^nx_ky_k.
\end{align*}
Therefore $\phi$ maps $U_1^*S^n$, seen as a subset of $B^{2n}(\sqrt{2})$, bijectively to $Q^{n}-Q^{n-1}$.

\end{proof}

At this point, we want to show that the above compactification agrees with the more intrinsic one obtained by performing a symplectic cut on $T^*S^n$ with respect to the $S^1$-action induced by the co-geodesic flow. By definition, this is the flow induced by the Hamiltonian $H: T^*S^n\rightarrow \mathbb{R}$ defined by $H(p,q)=\frac{1}{2}|q|^2$.

\begin{proposition}\label{unitcut}
The closure of $\phi(U_1^*S^n)$ in $(Q^n,2\omega_{FS})$ is symplectomorphic to the symplectic cut along the level set $H^{-1}(\frac{1}{2})$ of the Hamiltonian $H$ inducing the co-geodesic flow on $T^*S^n$. In particular, the symplectic cut of $U_1^*S^n$ with respect to the co-geodesic flow is symplectomorphic to $Q^n$ with the restriction of $\omega_{FS}$.
\end{proposition}
\begin{proof}
The closure of $\phi(U_1^*S^n)$ is just the $\phi$-image of $U_1^*S^n$ and its boundary $S^*S^n$. Consider the symplectic cut defined by the Hamiltonian $F(z)=\frac{1}{2}|z|^2$, on $F^{-1}((-\infty,1])$. The map $\phi$, restricted to  $\overline{B}^{2n+2}(\sqrt{2})$,  is exactly the quotient map of that symplectic cut.\par 
Therefore the closure of $\phi(U_1^*S^n)$ is symplectomorphic to the symplectic cut defined by the restriction of $F$ to the closure of $U_1^*S^n$, seen as a subset of $\overline{B}^{2n+2}(\sqrt{2})$. Notice that the two Hamiltonians differ by a constant, i.e.
\[F|_{T^*S^n}=H+\frac{1}{2}\]
and therefore the symplectic cuts of $T^*S^n$ on $F^{-1}(1)$ and $H^{-1}(\frac{1}{2})$ are symplectomorphic since the corresponding Hamiltonians induce the same flow.
\end{proof}

In the next step, we will use the above compactification of $U_1^*S^n$ and the fact that $U_1^*S^n$ double covers $U_1^*\mathbb{R}P^n$ to derive a compactification for $T^*\mathbb{R}P^n$. 

\begin{corollary}\label{branchedcovering}
By performing a symplectic cut along the boundary of $U_1^*\mathbb{R}P^n$ we get a manifold diffeomorphic to $\mathbb{C}P^n$, 
equipped with a symplectic form $\overline{\omega}_n$. In particular, the symplectic cut extends the symplectic double cover $U_1^*S^n\xrightarrow{\pi'} U_1^*\mathbb{R}P^n$ to a symplectic branched double cover $Q^n\xrightarrow{\pi}\mathbb{C}P^n$, branched along the quadric $Q^{n-1}$ and satisfying $\pi^*\overline{\omega}_{n}=\omega_{FS}|_{Q^n}$.
\end{corollary}

\begin{proof}
First, let us examine the branched covering. We can define a projection map $\pi: \mathbb{C}P^{n+1}-[0:\dots:0:1]\rightarrow \mathbb{C}P^n$ simply by \[[z_0:\dots:z_n:z_{n+1}]\xrightarrow{\pi} [z_0:\dots:z_n].\]
When restricted to $Q^n$ this map gives a double cover $Q^n\rightarrow \mathbb{C}P^n$, branched along $Q^{n-1}\subset\mathbb{C}P^n$. These maps fit in the following commutative diagram:
\[
\begin{tikzcd}
U_1^*S^n \arrow[r, "i", hook] \arrow[d, "\pi'"'] & Q^n \arrow[d, "\pi"] \\
U_1^*\mathbb{R}P^n \arrow[r, "i'", hook]      & \mathbb{C}P^n       
\end{tikzcd}
\]
The deck transformation of $\pi$ is given by 

\[\delta:[z_0:\dots:z_{n+1}]\rightarrow [z_0:\dots:-z_{n+1}].\]
The inclusion $U^*_1 S^n\xhookrightarrow{i} Q^n$ commutes with $\delta$ and with the antipodal action, i.e. $\delta\circ i (p,q)=i(-p,-q)$, therefore we can push down the symplectic form of $Q^n-Q^{n-1}$ to $\mathbb{C}P^{n}-Q^{n-1}$. Analogously to the inclusion $U_1^*S^n\xhookrightarrow{i}Q^n$, the inclusion $U_1^*\mathbb{R}P^n\xhookrightarrow{i'} \mathbb{C}P^n$ extends to a diffeomorphism from the symplectically cut cotangent disc bundle to $\mathbb{C}P^n$. 
\end{proof}

\begin{remark}
While the above construction is very natural from an algebro-geometric perspective, unfortunately, the symplectic form $\overline{\omega}_n$ we obtain after the symplectic cut on $U_1^*\mathbb{R}P^n$ is \textbf{not} the standard Fubini-Study form, even up to scaling. Since $\pi$ is a holomorphic map (concerning the standard complex structures), if it also preserved the standard symplectic forms it would have to be an isometry and this would lead to a contradiction since $Q^n$ is orthogonal to the pre-images of $\pi$ only along $Q^{n-1}$. 

\end{remark}

We turn now our attention to dimension $4$. Here, we can use the deep and classical results coming from pseudoholomorphic curve techniques to correct the form $\overline{\omega}_2$ and provide the desired compactification of $U_1^*\mathbb{R}P^2$.

\begin{proposition}\label{compactprop}
The form $\overline{\omega}_2$ on $\mathbb{C}P^2$ is diffeomorphic to the form $2\omega_{FS}$. In addition, the diffeomorphism can be chosen to fix both the standard $\mathbb{R}P^2$ and the quadric $Q^1$.
\end{proposition}

\begin{proof}
The symplectic manifold $(\mathbb{C}P^2,\overline{\omega}_2)$ contains a symplectic sphere with positive self-intersection, namely the quadric $Q^1$. Thus, by a fundamental theorem of McDuff (Theorem 9.4.1 in \cite{JMS}), there exists a diffeomorphism $f:\mathbb{C}P^2\rightarrow \mathbb{C}P^2$ such that $f^*(2\omega_{FS})=\overline{\omega}_2$, since the symplectic sphere $Q^1$ has the same area with respect to both forms. We will modify $f$ in two stages: Firstly, we will correct the image of the standard real projective plane $L$ and secondly, we will correct the image of the quadric $Q^1$.\par

For the first step, one can use Theorem \ref{laguniqcp2}. By composing  $f$ with a symplectomorphism mapping $f(L)$ to $L$, we may, and will, assume that $f(L)=L$.\par

The second step follows by composing $f$ with the Hamiltonian symplectomorphism given in Lemma \ref{sphereuniq}; as a result, $f$ fixes not only $L$ but also $Q^1$.
\end{proof}

\begin{remark}
For $n>2$, we were unable to determine whether $\overline{\omega}_n$ is diffeomorphic to $2\omega_{FS}$ or not.
\end{remark}

Let us conclude by mentioning that the above construction also works when one wants to compactify the radius $r$ cotangent disc bundles of either $U_r^*S^n$ or $U_r^*\mathbb{R}P^n$. The only slight modification that one has to do is to first pass from $U_r^*(S^n(1))$ to $U_{\sqrt{r}}^*(S^n(\sqrt{r}))$ and then proceed exactly as above. The reason for this is because we need an \say{evened} disc bundle, i.e. the radius of the fibers to match the radius of the $n$-sphere, to identify the action induced by the co-geodesic flow on the boundary of $U_{\sqrt{r}}^*S^n(\sqrt{r})$ with the scalar action on $S^{2n+1}(\sqrt{2r})$. Having that in mind, one obtains:

\begin{proposition}$\label{almostcompactif}$
The symplectic cut along $(U^*_r(S^n(1)),\omega_{std})$ is symplectomorphic to $(Q^n,2r\omega_{FS})$. In addition, the symplectic cut along $(U^*_r\mathbb{R}P^n,\omega_{std})$ gives rise to the symplectic branched cover $(Q^n,2r\omega_{FS})\rightarrow (\mathbb{C}P^n,\overline{\omega}_{n,r})$. In the case $n=2$, the form $\overline{\omega}_{2,r}$ is symplectomorphic to $2r\omega_{FS}$.
\end{proposition}



\section{Arnold's Conjecture for $\mathbb{R}P^2$}

The final ingredient we need is a statement about the homotopy type of the group of compactly supported symplectomorphisms of $T^*\mathbb{R}P^2$, which we denote by $\text{Symp}_c(T^*\mathbb{R}P^2)$. The relevant result is proved by Evans in \cite{Emcg}, using equivariant versions of arguments that Seidel used in \cite{S} to study $\text{Symp}_c(T^*S^2)$.  \par 

To state the above results, one has to introduce the notion of a \textit{generalized Dehn-Twist} along a Lagrangian sphere. Roughly speaking, one can define a map $\tau: T^*S^n\rightarrow T^*S^n$ that is equivariant with respect to the antipodal action, restricts to the antipodal map on the zero section, and is the identity away from the zero section. This definition generalizes the usual Dehn Twist that one encounters in $2$-dimensional topology. Of course, since $\tau$ is antipodal-equivariant, it descends to a diffeomorphism $\bar{\tau}$ of $T^*\mathbb{R}P^n$. \par 
Even though it is somewhat technical to give an explicit formula for $\tau$ and $\bar{\tau}$, they can be constructed to give symplectomorphisms of $T^*S^n$ and of $T^*\mathbb{R}P^n$ respectively. In particular, since $\tau,\bar{\tau}$ can be made to have support in an arbitrarily small neighborhood of the zero section, by Weinstein's Lagrangian neighborhood Theorem, one can perform a Dehn-twist along any Lagrangian $S^n$ or $\mathbb{R}P^n$. Lagrangian Dehn Twists along $S^n$ were first explored by Seidel in \cite{S} to give examples of symplectomorphisms, namely squares of Dehn-Twists, that are smoothly but not symplectically isotopic to the identity. \par   
The relevant result we need is:

\begin{proposition}\emph{(Section 3 in \cite{Emcg})}\label{dehntwist}
The Dehn-Twist $\bar{\tau}$ along the zero section of $T^*\mathbb{R}P^2$ generates $\pi_0(\text{Symp}_c(T^*\mathbb{R}P^2))$. In other words, any compactly supported symplectomorphism $\phi$ can be written as
\[\phi=\bar{\tau}^n\circ\eta\] 
for some $n\geq 0$ and $\eta$ a Hamiltonian symplectomorphism.
\end{proposition}

Combining the above Proposition with the results of the previous sections, we can finally prove the main theorem of this paper:

\begin{Theorem}\emph{(Arnold's conjecture for $\mathbb{R}P^2$)}\label{maintheorem}
Let $L$ be an embedded Lagrangian submanifold of $T^*\mathbb{R}P^2$ which is diffeomorphic to $\mathbb{R}P^2$. There exists a compactly supported Hamiltonian isotopy $\{H_t\}_{t\in [0,1]}$ such that $H_1(L)=0_s$ where $0_s$ is the zero-section.
\end{Theorem}
\begin{proof}
Since $L$ is compact we can assume that, without loss of generality, it is contained in $U_1^*\mathbb{R}P^2$. Following Proposition \ref{compactprop}, we compactify this disc bundle to $(\mathbb{C}P^2,2\omega_{FS})$, where the boundary of the bundle is mapped to a degree $2$ symplectic sphere $S$. Thus, we consider the Lagrangians $L$ and $0_s$ as living in the complement of $S$. Now we can apply Lemma \ref{sphereuniq} and Theorem \ref{laguniqcp2}, as in the proof of Proposition \ref{compactprop}, to find a symplectomorphism $\psi$ of $\mathbb{C}P^2$ such that $\psi(L)=0_s$ and $\psi=id$ in a neighborhood of $S$. Therefore $\psi$ gives a compactly supported symplectomorphism of $U_1^*\mathbb{R}P^2$ mapping $L$ to $0_s$.\par

Using Proposition \ref{dehntwist} we can factor $\psi$ as $\psi=\bar\tau^n\circ \eta$. Since $\psi(L)=0_s$ and $\bar\tau^n(0_s)=0_s$ we have that $\eta(L)=0_s$ thus $\eta$ is the desired Hamiltonian symplectomorphism.

\end{proof}


\printbibliography

\end{document}